\documentclass[reqno,10pt,centertags]{amsart}
\usepackage{amsmath,amsthm,amscd,amssymb,latexsym,upref,enumerate,stmaryrd}
\usepackage{amsfonts,mathrsfs}
\usepackage{esint}
\usepackage{color}
\usepackage{hyperref} 

\newcommand*{\mailto}[1]{\href{mailto:#1}{\nolinkurl{#1}}}
\newcommand{\arxiv}[1]{\href{http://arxiv.org/abs/#1}{arXiv:#1}}

\newcommand{\bbC}{{\mathbb{C}}}

\newcommand{\bbN}{{\mathbb{N}}}

\newcommand{\bbR}{{\mathbb{R}}}

\newcommand{\cB}{{\mathcal B}}

\newcommand{\cE}{{\mathcal E}}

\newcommand{\cH}{{\mathcal H}}

\newcommand{\cW}{{\mathcal W}}
\newcommand{\cX}{{\mathcal X}}

\newcommand{\gq}{\mathfrak q}

\newcommand{\e}{{\varepsilon}}

\DeclareMathOperator{\diam}{diam}

\DeclareMathOperator{\dist}{dist}
\DeclareMathOperator{\supp}{supp}

\DeclareMathOperator{\dom}{dom}

\renewcommand{\Re}{\text{\rm Re}}

\renewcommand{\arg}{\text{\rm arg}}

\newcommand{\beq}{\begin{equation}}
\newcommand{\enq}{\end{equation}}

\newcommand{\no}{\notag}
\newcommand{\lb}{\label}

\newcommand{\ol}{\overline}

\newcommand{\bi}{\bibitem}

\let\geq\geqslant
\let\leq\leqslant

\makeatletter
\def\theequation{\@arabic\c@equation}

\allowdisplaybreaks 
\numberwithin{equation}{section}

\newtheorem{theorem}{Theorem}[section]
\newtheorem{proposition}[theorem]{Proposition}

\newtheorem{definition}[theorem]{Definition}
\newtheorem{hypothesis}[theorem]{Hypothesis}

\theoremstyle{remark}
\newtheorem{remark}[theorem]{Remark}

\begin{document}

\title[Stability of Square Root Domains]{Stability of Square Root Domains Associated with 
Elliptic Systems of PDEs on Nonsmooth Domains}

\author[F.\ Gesztesy]{Fritz Gesztesy}
\address{Department of Mathematics,
University of Missouri, Columbia, MO 65211, USA}
\email{\mailto{gesztesyf@missouri.edu}}
\urladdr{\url{http://www.math.missouri.edu/personnel/faculty/gesztesyf.html}}

\author[S.\ Hofmann]{Steve Hofmann}
\address{Department of Mathematics,
University of Missouri, Columbia, MO 65211, USA}
\email{\mailto{hofmanns@missouri.edu}}
\urladdr{\url{http://www.math.missouri.edu/~hofmann/}}

\author[R.\ Nichols]{Roger Nichols}
\address{Mathematics Department, The University of Tennessee at Chattanooga, 
415 EMCS Building, Dept. 6956, 615 McCallie Ave, Chattanooga, TN 37403, USA}
\email{\mailto{Roger-Nichols@utc.edu}}
\urladdr{\url{https://www.utc.edu/faculty/roger-nichols/}}

\thanks{S.H. was supported by NSF grant DMS-1101244} 
\thanks{R.N. gratefully acknowledges support from an AMS--Simons Travel Grant.}
\thanks{To appear in {\it J. Differential Equations}.} 
\date{\today}

\begin{abstract} 
We discuss stability of square root domains for uniformly 
elliptic partial differential operators $L_{a,\Omega,\Gamma} = -\nabla\cdot a \nabla$ in 
$L^2(\Omega)$, with mixed boundary conditions on $\partial \Omega$, with respect to 
additive perturbations. We consider open, bounded, and connected sets $\Omega \in \bbR^n$, 
$n \in \bbN \backslash\{1\}$, that satisfy the interior corkscrew condition and prove stability of 
square root domains of the operator $L_{a,\Omega,\Gamma}$ with respect to additive potential 
perturbations $V \in L^p(\Omega) + L^{\infty}(\Omega)$, $p>n/2$.  

Special emphasis is put on the case of uniformly elliptic systems with mixed boundary conditions. 
\end{abstract}

\maketitle

\section{Introduction}  \lb{s1}

The aim of this note is to provide applications of a recently developed abstract approach to the 
stability of square root domains of non-self-adjoint operators with respect to additive perturbations 
to elliptic partial differential operators with mixed boundary conditions on a class of open, bounded, 
connected sets $\Omega \subset \bbR^n$, $n \in \bbN\backslash\{1\}$, that satisfy the corkscrew 
condition (and hence go beyond bounded Lipschitz domains).     

More precisely, if $T_0$ is an appropriate non-self-adjoint m-accretive operator in a separable, complex  
Hilbert space $\cH$, we developed an abstract approach in \cite{GHN12} to determine conditions 
under which non-self-adjoint additive perturbations $W$ of $T_0$ yield the stability of square root 
domains in the form  
\begin{equation}
\dom\big((T_0 + W)^{1/2}\big) = \dom\big(T_0^{1/2}\big).    \lb{1.1}
\end{equation} 
In fact, driven by applications to PDEs, we were particularly interested in the following variant of 
this stability problem for square root domains with respect to additive perturbations: if $T_0$ is an  
appropriate non-self-adjoint operator for which it is known that Kato's square root problem in the 
following abstract form, that is, 
\begin{equation}
\dom\big(T_0^{1/2}\big) = \dom\big((T_0^*)^{1/2}\big)    \lb{1.2} 
\end{equation}
is valid, for which non-self-adjoint additive perturbations $W$ of $T_0$ can one conclude that also 
\begin{equation}
\dom\big((T_0 + W)^{1/2}\big) = \dom\big(T_0^{1/2}\big) = \dom\big((T_0^*)^{1/2}\big) 
= \dom\big(((T_0 + W)^*)^{1/2}\big)     \lb{1.3}
\end{equation}
holds? 

Without going into details at this point we note that $T_0 + W$ will be viewed as a form sum of $T_0$ and 
$W$. 

Formally speaking, the role of the operator $T_0$ in $\cH$ in this note will be played by 
$L_{a,\Omega,\Gamma}$, an m-sectorial realization of the uniformly elliptic differential expression 
in divergence form, $-\nabla\cdot a \nabla$, in $L^2(\Omega)$, with  $\Omega \subset \bbR^n$, 
$n \in \bbN\backslash\{1\}$ an open, bounded, connected set that satisfies the corkscrew condition, 
and the coefficients $a_{j,k}$, $1\leq j,k\leq n$, assumed to be essentially bounded. Moreover, 
$L_{a,\Omega,\Gamma}$ is constructed in such a manner via quadratic forms so that it satisfies 
a Dirichlet boundary condition along the closed (possibly empty) subset 
$\Gamma \subseteq \partial\Omega$ and a Neumann boundary condition on the remainder of 
the boundary, $\partial \Omega\backslash \Gamma$. The additive perturbation $W$ of $T_0$ then 
is given by a potential term $V$, that is, by an operator of multiplication in $L^2(\Omega)$ by an 
element 
\begin{equation} 
V \in L^p(\Omega) + L^{\infty}(\Omega) \, \text{ for some $p > n/2$.} 
\end{equation} 
(For simplicity we will not consider the one-dimensional case $n=1$ in this note as that has been 
separately discussed in \cite{GHN13}.)

In fact, we will go a step further and consider uniformly elliptic systems in $L^2(\Omega)^N$, 
$N \in \bbN$, where $T_0$ in $\cH$ is represented by $\mathbf{L}_{a,\Omega,\mathbb{G}}$ in $L^2(\Omega)^N$, the m-sectorial realization of the $N\times N$ matrix-valued differential 
expression $\mathbf{L}$ which acts as 
\begin{equation} \lb{1.4}
\mathbf{L} u = - \Bigg(\sum_{j,k=1}^n\partial_j\Bigg(\sum_{\beta = 1}^N 
a^{\alpha,\beta}_{j,k}\partial_k u_\beta\Bigg)
\Bigg)_{1\leq\alpha\leq N},\quad u=(u_1,\dots,u_N), 
\end{equation}
with $a_{j,k}^{\alpha,\beta}\in L^{\infty}(\Omega)$, $1\leq j,k\leq n$, $1\leq \alpha,\beta\leq N$. 
Here $\mathbb{G}$ represents the collection $\mathbb{G}=(\Gamma_1,\Gamma_2,\ldots,\Gamma_N)$, 
with $\Gamma_{\alpha}\subseteq \partial \Omega$ a closed (possibly empty) subset 
of $\partial \Omega$, and intuitively, $\mathbf{L}$ acts on vectors $u = (u_1, u_2,\dots,u_N)$, where each component $u_{\alpha}$ formally satisfies a Dirichlet boundary condition along $\Gamma_{\alpha}$ and a 
Neumann condition along the remainder of the boundary, $\partial \Omega\backslash \Gamma_{\alpha}$, 
$1 \leq \alpha \leq N$.  
The additive perturbation $W$ of $T_0$ then corresponds to an $N \times N$ 
matrix-valued operator of multiplication in $L^2(\Omega)^N$ of the form 
\begin{equation}\lb{1.5} 
 (\mathbf{V}f)_{\alpha} = \sum_{\beta=1}^N V_{\alpha,\beta} f_{\beta},\quad 1\leq \alpha \leq N, \;  
f\in \dom(\mathbf{V}) = \big\{ f\in L^2(\Omega)^N\, \big|\, \mathbf{V}f \in L^2(\Omega)^N\big\}.
\end{equation}
with 
\begin{equation} 
V_{\alpha,\beta} \in L^{p} (\Omega) + L^{\infty} (\Omega) \, \text{ for some $p> n/2$, 
$1 \leq \alpha, \beta \leq N$.} 
\end{equation}  

The considerable amount of literature on Kato's square root problem in the concrete case where 
$T_0$ represents a uniformly elliptic differential operator in divergence form $-\nabla\cdot a \nabla$ 
in $L^2(\Omega)$ with various boundary conditions on $\partial \Omega$, has been reviewed in great 
detail in \cite{GHN12}. Thus, in this note we now confine ourselves to refer, for instance, in addition to 
\cite{AAM10}, \cite{AAM10a}, \cite{ABHDR12}, \cite{AHLLMT01}, \cite{AHLMT02}, \cite{AHLT01}, \cite{AHMT01}, \cite{AKM06}, \cite{AT92}, \cite{AT98}, \cite{AT03}, \cite{CMM82}, \cite{EHDT13a}, \cite{EHDT13b}, \cite{HLM02}, \cite{Mc85}, \cite{Mc90}, \cite{Ya87}, and the references cited in 
these sources. 

The starting point for this note was a recent paper by 
Egert, Haller-Dintelmann, and Tolksdorf \cite{EHDT13a} (cf.\ Theorem \ref{t2.5}), which permits us 
to go beyond the class of strongly Lipschitz domains considered in \cite{GHN12} and now consider 
open, bounded, and connected sets $\Omega \in \bbR^n$ that satisfy the interior corkscrew condition.
In Section \ref{s2} we first consider uniformly elliptic partial differential operators with mixed boundary conditions on $\Omega$, closely following \cite{EHDT13a}, and subsequently study the quadratic 
forms associated with $L_{a,\Omega,\Gamma}$ and $V$. We then prove stability of square root 
domains of the operator $L_{a,\Omega,\Gamma}$ with respect to additive perturbations 
$V \in L^p(\Omega) + L^{\infty}(\Omega)$, $p>n/2$, for this more general class of domains $\Omega$. 
The extension of these results to elliptic systems governed by \eqref{1.4} and perturbed by the 
matrix-valued potentials in \eqref{1.5} then is the content of Section \ref{s3}. 
 
Finally, we briefly summarize some of the notation used in this paper: 
Let $\cH$ be a separable, complex Hilbert space with scalar product 
(linear in the second argument) and norm denoted by $(\cdot,\cdot)_{\cH}$  and 
$\|\cdot\|_{\cH}$, respectively. 
Next, if $T$ is a linear operator mapping (a subspace of) a Hilbert space into another, then 
$\dom(T)$ denotes the domain of $T$. 
The closure of a closable operator $S$ is denoted by $\ol S$. 
The form sum of two (appropriate) operators $T_0$ and $W$ is abbreviated by $T_0 +_{\gq} W$. 

The Banach space of bounded linear operators on a separable complex Hilbert 
space $\cH$ is denoted by $\cB(\cH)$. The notation $\cX_1\hookrightarrow \cX_2$ is used 
for the continuous embedding of the Banach space $\cX_1$ into the Banach space $\cX_2$. 

If $n\in \bbN$ and $\Omega\subset \bbR^n$ is a bounded set, then 
$\diam(\Omega)=\sup_{x,y\in \Omega}|x-y|$ denotes the diameter of $\Omega$.  We use $m_{\ell,n}$ 
to denote the $\ell$-dimensional Hausdorff measure on $\bbR^n$ (and hence $m_{n,n}$, also 
denoted by $|\cdot|$, represents the $n$-dimensional Lebesgue measure on $\bbR^n$). 
If $x\in \bbR^n$ and $r>0$, then $B(x,r)$ denotes the open ball of radius $r$ centered at $x$. 
In addition, $I_n$ denotes the $n\times n$ identity matrix in $\bbC^n$, and the set of $k\times \ell$ 
matrices with complex-valued entries is denoted by $\mathbb{C}^{k \times \ell}$. Finally, we abbreviate $L^p(\Omega; d^n x) := L^p(\Omega)$ and $L^p(\Omega, \bbC^N; d^n x) := L^p(\Omega)^N$, 
$N \in \bbN$.

\section{Elliptic Partial Differential Operators \\ with Mixed Boundary Conditions}  \lb{s2}

In this section we discuss stability of square root domains for uniformly 
elliptic partial differential operators $L_{a,\Omega,\Gamma} = -\nabla\cdot a \nabla$ in 
$L^2(\Omega)$, with mixed 
boundary conditions on $\partial \Omega$, with respect to additive perturbations in \cite{GHN12},  
by employing a recent result due to Egert, Haller-Dintelmann, and Tolksdorf \cite{EHDT13a} 
(recorded in Theorem \ref{t2.5} below). This permits us to go beyond 
the class of strongly Lipschitz domains considered in \cite{GHN12} and now consider open, bounded,  
and connected sets $\Omega \in \bbR^n$ that satisfy the interior corkscrew condition. We then prove 
stability of square root domains of the operator $L_{a,\Omega,\Gamma}$ with respect to additive 
potential perturbations $V \in L^p(\Omega) + L^{\infty}(\Omega)$, $p>n/2$, for this more general 
class of domains $\Omega$. 

We start with the following definitions:

\begin{definition}\lb{d2.1}
Let $n\in \bbN\backslash\{1\}$. \\
$(i)$ A nonempty, bounded, open, and connected set $\Omega\subset \bbR^n$ is said to satisfy the interior corkscrew condition if there exists a constant $\kappa \in (0,1)$ with the property that for each $x\in \ol{\Omega}$ and each $r\in (0,\diam(\Omega))$, there exists a point $y\in \ol{B(x,r)}$ such that $\ol{B(y,\kappa r)}\subseteq \Omega$. \\
$(ii)$ Let $\Omega$ be a nonempty, proper, open subset of $\bbR^n$. 
One calls $\Omega$ a Lipschitz domain if for every $x_0\in\partial\Omega$ 
there exist $r>0$, a rigid transformation $T:\bbR^n\to\bbR^n$, 
and a Lipschitz function $\varphi:\bbR^{n-1}\to\bbR$ with the 
property that 
\begin{equation}
T\big(\Omega\cap B(x_0,r)\big)=T\big(B(x_0,r)\big)\cap \big\{(x',x_n)\in\bbR^{n-1}\times\bbR 
\, \big| \, x_n>\varphi(x')\big\}.
\end{equation}
\end{definition}

We recall our convention that $m_{\ell,n}$ denotes the $\ell$-dimensional Hausdorff measure 
(for the basics on Hausdorff measure, see, e.g., \cite[Ch.~2]{EG92}, \cite[Ch.~2]{Ro98}) 
and hence $m_{n,n} = |\cdot|$ represents $n$-dimensional Lebesgue measure on $\bbR^n$. 

The following proposition records some basic results in connection with the interior corkscrew condition.  

\begin{proposition}\lb{p2.2}
Let $n\in \bbN\backslash\{1\}$ and suppose that $\Omega\subset \bbR^n$ is nonempty, 
bounded, open, and connected.  Then the following items $(i)$ and $(ii)$ hold: \\
$(i)$  If $\Omega$ satisfies the interior corkscrew condition with constant $\kappa\in (0,1)$, then
\begin{equation}\lb{2.1}
\kappa^nr^n \leq |\Omega \cap B(x,r)| \leq r^n,\quad x\in \Omega,\; 0<r<\diam(\Omega).
\end{equation}
$(ii)$  If $\Omega$ is a bounded Lipschitz domain, then $\Omega$ satisfies the interior corkscrew condition.
\end{proposition}

\begin{definition}\lb{d2.3}
Suppose $n\in \bbN$ is fixed and $0<\ell\leq n$.  A non-empty Borel set $M\subseteq \bbR^n$ 
is an $\ell$-set if there exist constants $c_j=c_j(M)>0$, $j =1,2$, for which
\begin{equation}\lb{2.2}
c_1r^{\ell} \leq m_{\ell,n}(M \cap B(x,r)) \leq c_2r^{\ell},\quad x\in M,\; 0<r\leq 1.
\end{equation}
\end{definition}

One notes that $\ol{M}$ is an $\ell$-set if $M$ is an $\ell$-set and 
$m_{\ell,n}\big(\ol{M}\backslash M\big)=0$ in this case.

\begin{hypothesis} \lb{h2.4} 
Let $n\in \bbN\backslash\{1\}$. \\
$(i)$ Assume that $\Omega\subset \bbR^n$ is a nonempty, bounded, open, and connected set satisfying the interior corkscrew condition in Definition \ref{d2.1}\,$(i)$. \\
$(ii)$  Suppose $\Gamma\subset \partial\Omega$ is closed and for every 
$x\in \ol{\partial \Omega \backslash \Gamma}$, there exists an open neighborhood $U_x \subset \bbR^n$ 
and a bi-Lipschitz map $\Phi_x:U_x\rightarrow (-1,1)^n$ such that 
\begin{align}
& \Phi_x(x) = 0,\lb{2.3}\\
& \Phi_x(\Omega \cap U_x) = (-1,1)^{n-1}\times (-1,0),\lb{2.4}\\
& \Phi_x(\partial \Omega \cap U_x) = (-1,1)^{n-1}\times \{0\}.\lb{2.5}
\end{align}
$(iii)$  Suppose $\Gamma=\emptyset$ or $\Gamma$ is an $(n-1)$-set.\\
$(iv)$ Assume that $a: \Omega \rightarrow \bbC^{n\times n}$ is a Lebesgue measurable, matrix-valued function which is essentially bounded and uniformly elliptic, that is, there exist constants 
$0 < a_1 \leq a_2 < \infty$ such that for a.e.\ $x\in \Omega$, 
\begin{equation}\lb{2.6}
a_1 \|\xi\|_{\bbC^n}^2 \leq \Re\big[(\xi,a(x)\xi)_{\bbC^n}\big] \, \text{ and } \, 
|(\zeta, a(x)\xi)_{\bbC^n}|\leq a_2 \|\xi\|_{\bbC^n} \|\zeta\|_{\bbC^n}, \quad   
\text{$\xi,\zeta \in \bbC^n$.} 
\end{equation}
$(v)$ With $C_{\Gamma}^{\infty}(\Omega)$ defined by 
\begin{equation}\lb{2.7}
C_{\Gamma}^{\infty}(\Omega):= \{u|_{\Omega}\, | \, u\in C^{\infty}(\bbR^n),\, \dist(\supp(u),\Gamma)>0\},
\end{equation}
denote by $W^{1,2}_{\Gamma}(\Omega)$ the closure of $C_{\Gamma}^{\infty}(\Omega)$ in 
$W^{1,2}(\Omega)$, that is,
\begin{equation}\lb{2.8}
W^{1,2}_{\Gamma}(\Omega) = \ol{C_{\Gamma}^{\infty}(\Omega)}^{W^{1,2}(\Omega)},
\end{equation}
and introduce the densely defined, accretive, and closed sesquilinear form in $L^2(\Omega)$,  
\begin{equation}\lb{2.9}
\mathfrak{q}_{a,\Omega,\Gamma}(f,g) = \int_{\Omega}d^nx\, ((\nabla f)(x),a(x)(\nabla g)(x))_{\bbC^n},\quad f,g\in \dom(\mathfrak{q}_{a,\Omega,\Gamma}):=W^{1,2}_{\Gamma}(\Omega).
\end{equation}
We denote by $L_{a,\Omega,\Gamma}$ the m-sectorial operator in $L^2(\Omega)$ uniquely 
associated to $\mathfrak{q}_{a,\Omega,\Gamma}$. \\
$(vi)$ Suppose that $V:\Omega\rightarrow \bbC$ is $($Lebesgue$)$ measurable and factored according to 
\begin{equation}\lb{2.10}
V(x) = u(x) v(x), \quad v(x)=|V(x)|^{1/2}, \quad u(x) = e^{i\arg(V(x))} v(x) \, 
\text{ for a.e. $x\in \Omega$,}
\end{equation}
such that 
\begin{equation}\lb{2.11}
W^{1,2}_{\Gamma}(\Omega)\subseteq \dom(v).
\end{equation}
\end{hypothesis}

In the special case where $a(x)=I_n$ for a.e. $x\in \Omega$, with $I_n$ the $n\times n$ identity 
matrix in $\bbC^n$, we simplify notation and write 
\begin{equation}\lb{2.12}
L_{I_n,\Omega,\Gamma}=-\Delta_{\Omega,\Gamma}.
\end{equation}
Note that $-\Delta_{\Omega,\Gamma}$ is self-adjoint and non-negative.

For an example of a bounded, open, and connected set that satisfies the conditions $(i)$--$(iii)$ 
of Hypothesis \ref{h2.4} and is not Lipschitz, see \cite[Figure 1]{EHDT13a}.  One notes that Hypothesis \ref{h2.4}\,$(i)$ permits {\it inward-pointing} cusps.

Formally speaking, the operator $L_{a,\Omega,\Gamma}$ is of uniform elliptic divergence form 
$L_{a,\Omega,\Gamma}=-\nabla\cdot a \nabla$, satisfying a Dirichlet boundary condition along 
$\Gamma$ and a Neumann (or, natural) boundary condition on the remainder of the boundary, 
$\partial \Omega\backslash \Gamma$. 

The quadratic form $\gq^{}_V$ in $L^2(\Omega)$, uniquely associated with $V$, is defined by
\begin{equation}
\gq^{}_V(f,g) = \big(v f, e^{i \arg(V)} v g\big)_{L^2(\Omega)}, \quad 
f, g \in \dom(\gq^{}_V) = \dom(v).     \lb{2.13} 
\end{equation}
Under appropriate assumptions on $V$ (see Hypotheses \ref{h2.6} and \ref{h2.7} below), the 
form sum of $\gq^{}_{a,\Omega,\Gamma}$ and $\gq^{}_V$ will define a sectorial form on $W^{1,2}_{\Gamma}(\Omega)$ 
and the operator uniquely associated to $\gq^{}_{a,\Omega,\Gamma} + \gq^{}_V$ will be denoted by 
$L_{a,\Omega,\Gamma} +_{\gq} V$ (see also the paragraph following \cite[eq.~(A.42)]{GHN12}).

The principal aim of this section is to prove stability of square root domains in the form
\begin{align}\lb{2.14}
\dom\big((L_{a,\Omega,\Gamma} +_{\gq} V)^{1/2} \big) = \dom\big(L_{a,\Omega,\Gamma}^{1/2}\big)= W^{1,2}_{\Gamma}(\Omega),
\end{align}
under appropriate (integrability)  assumptions on $V$, thereby extending the recent results on stability of square root domains obtained in \cite{GHN12} to the setting of certain classes of non-Lipschitz domains with mixed boundary conditions as discussed in \cite{EHDT13a}.  As a basic input, we rely on the following result which is Theorem 4.1 in \cite{EHDT13a}.

\begin{theorem}[Egert--Haller-Dintelmann--Tolksdorf \cite{EHDT13a}]\lb{t2.5}
Assume items $(i)$--$(v)$ of Hypothesis \ref{h2.4}. Then
\begin{equation}\lb{2.15}
\dom\big(L_{a,\Omega,\Gamma}^{1/2} \big) = \dom\big((L_{a,\Omega,\Gamma}^*)^{1/2} \big) = W^{1,2}_{\Gamma}(\Omega).
\end{equation}
\end{theorem}

Next, we introduce various hypotheses corresponding to the potential coefficient $V$.

\begin{hypothesis}\lb{h2.6}
Let $n\in\bbN\backslash\{1\}$, assume that $\Omega \subseteq \bbR^n$ is nonempty and 
open, and let $V\in L^p(\Omega) + L^\infty(\Omega)$ for some $p>n/2$. 
\end{hypothesis}

In addition, we also discuss the critical $L^p$-index $p = n/2$ for $V$ for $n \geq 3$:

\begin{hypothesis}\lb{h2.7}
Let $n \in \bbN \backslash \{1,2\}$, assume that $\Omega \subseteq \bbR^n$ is 
nonempty and open, and let $V\in L^{n/2}(\Omega) + L^\infty(\Omega)$. 
\end{hypothesis}

Here $V\in L^q(\Omega) + L^\infty(\Omega)$ means as usual that $V$ permits a decomposition 
$V = V_q + V_{\infty}$ with $V_q \in L^q(\Omega)$ for some $q \geq 1$ and 
$V_{\infty} \in L^\infty(\Omega)$. 

\begin{theorem}\lb{t2.8}
Assume Hypotheses \ref{h2.4} and \ref{h2.6}.  Then the following items $(i)$ and $(ii)$ hold: \\
$(i)$ $V$ is infinitesimally form bounded with respect to $-\Delta_{\Omega,\Gamma}$, and there exist constants $M>0$ and $\e_0>0$ such 
that 
\begin{equation}\lb{2.16}
\begin{split}
\big\| |V|^{1/2}f \big\|_{L^2(\Omega)}^2\leq \e\big\|(-\Delta_{\Omega,\Gamma})^{1/2}f \big\|_{L^2(\Omega)}^2
+ M\e^{-n/(2p-n)}\|f\|_{L^2(\Omega)}^2,&\\
 f\in W^{1,2}_{\Gamma}(\Omega),\; 0<\e<\e_0.&
 \end{split}
\end{equation}
$(ii)$  $V$ is infinitesimally form bounded with respect to $L_{a,\Omega,\Gamma}$ and 
\begin{equation}\lb{2.17} 
\begin{split}
\big\| |V|^{1/2}f \big\|_{L^2(\Omega)}^2\leq \e \Re[\gq^{}_{a,\Omega,\Gamma}(f,f)] 
+ M a_1^{-n/(2p-n)}\e^{-n/(2p-n)}\|f\|_{L^2(\Omega)}^2,& \\
f\in W^{1,2}_{\Gamma}(\Omega),\; 0<\e< a_1^{-1}\e_0. &
\end{split}
\end{equation}
The form sum $L_{a,\Omega,\Gamma} +_{\gq} V$ is an m-sectorial operator which satisfies
\begin{align}
\begin{split}
& \dom\big((L_{a,\Omega,\Gamma} +_{\gq} V)^{1/2} \big) = \dom\big(((L_{a,\Omega,\Gamma} +_{\gq} V)^*)^{1/2} \big)   \\
& \quad = \dom\big(L_{a,\Omega,\Gamma}^{1/2}\big) = \dom\big((L_{a,\Omega,\Gamma}^*)^{1/2}\big) 
= W^{1,2}_{\Gamma}(\Omega).      \lb{2.18}
\end{split} 
\end{align}
\end{theorem}
\begin{proof}
For notational simplicity, and without loss of generality, we put the $L^\infty$-part $V_{\infty}$ of 
$V$ equal to zero for the remainder of this proof. Under Hypothesis \ref{h2.4}, there exists an extension operator $\cE$ satisfying 
\begin{equation}\lb{2.19}
(\cE f)(x)=f(x) \, \text{ for a.e.\ }\, x\in \Omega,\; f\in L^2(\Omega),
\end{equation}
with
\begin{align}
\begin{split} 
&\cE:W^{1,2}_{\Gamma}(\Omega)\rightarrow W^{1,2}(\bbR^n),    \\
&\|\cE f\|_{W^{1,2}(\bbR^n)}^2\leq C_1\|f\|_{W^{1,2}(\Omega)}^2, 
\quad f\in W^{1,2}(\Omega),    \lb{2.20}
\end{split} 
\end{align}
and
\begin{align}
\begin{split} 
&\cE:L^2(\Omega)\rightarrow L^2(\bbR^n),    \\
&\|\cE f\|_{L^2(\bbR^n)}^2\leq C_2\|f\|_{L^2(\Omega)}^2, 
\quad f\in L^2(\Omega),    \lb{2.21}
\end{split} 
\end{align}
for some constants $C_j>0$, $j=1,2$ (cf., e.g., \cite[Lemma 3.3]{ABHDR12}, \cite[Lemma 3.4]{ER12}).

Let $V_{\text{ext}}$ and $v_{\text{ext}}$ denote the extensions of $V$ and $v$, respectively, to 
all of $\bbR^n$ defined by setting $V_{\text{ext}}$ and $v_{\text{ext}}$ identical to zero on 
$\bbR^n\backslash\Omega$.  Evidently, $V_{\text{ext}}\in L^p(\bbR^n)$, so there exists a 
constant $M>0$ for which (cf., e.g., \cite[Lemma 3.7]{GHN12})
\begin{equation}\lb{2.22}
\begin{split}
\big\|v_{\text{ext}}f\big\|_{L^2(\bbR^n)}^2\leq \e \big\|(-\Delta)^{1/2} f\big\|_{L^2(\bbR^n)}^2
+ M\e^{-n/(2p-n)}\|f\|_{L^2(\bbR^n)}^2,&\\
f\in W^{1,2}(\bbR^n),\; \e>0.&
\end{split}
\end{equation}
Consequently, using \eqref{2.19}--\eqref{2.22}, one estimates
\begin{align}
\|vf\|_{L^2(\Omega)}^2&=\|v_{\text{ext}}\cE f\|_{L^2(\bbR^n)}^2\no\\
&\leq \varepsilon_1 \big\|(-\Delta)^{1/2} \cE f\big\|_{L^2(\bbR^n)}^2+M
\varepsilon_1^{-n/(2p-n)}\|\cE f\|_{L^2(\bbR^n)}^2\no\\
&=\varepsilon_1 \|\nabla \cE f\|_{L^2(\bbR^n)^n}^2+M\varepsilon_1^{-n/(2p-n)}\|\cE f\|_{L^2(\bbR^n)}^2\no\\
&\leq \varepsilon_1 C_1\|\nabla f\|_{L^2(\Omega)^n}^2+\big(\varepsilon_1 C_1 
+ C_0 M\epsilon^{-n/(2p-n)}\big)\|f\|_{L^2(\Omega)}^2\no\\
&\leq \varepsilon_1 C_1\|\nabla f\|_{L^2(\Omega)^n}^2 
+ \big(C_1+C_0M\big)\varepsilon_1^{-n/(2p-n)}\|f\|_{L^2(\Omega)}^2,\lb{2.23}\\
&\hspace*{4.5cm}f\in W^{1,2}_{\Gamma}(\Omega), \; 0 < \varepsilon_1 < 1.\no
\end{align}
To obtain the first term in the second equality above, we applied the 2nd representation theorem (cf., e.g., \cite[VI.2.23]{Ka80}) to the non-negative, self-adjoint operator $-\Delta$ on $W^{2,2}(\bbR^n)$ in 
$L^2(\bbR^n)$. The form bound in \eqref{2.16} now follows by choosing $\e=\varepsilon_1 C_1$ throughout \eqref{2.23} and noting that
\begin{equation}\lb{2.24}
\|\nabla f\|_{L^2(\Omega)^n}^2=\big\|(-\Delta_{\Omega,\Gamma})^{1/2}f\big\|_{L^2(\Omega)}^2, 
\quad f\in W^{1,2}_{\Gamma}(\Omega),
\end{equation}
by another application of the 2nd representation theorem (see \eqref{2.9} with $a(\,\cdot\,)=I_n$), 
proving item $(i)$.

In view of \eqref{2.15} and \eqref{2.16}, one notes that the hypotheses of \cite[Theorem 3.6]{GHN12} are met.  The statements in item $(ii)$ thus follow from a direct application of \cite[Theorem 3.6]{GHN12}.
\end{proof}

\begin{remark} \lb{r2.9} 
The proof of Theorem \ref{t2.8} follows the proof of \cite[Theorem 3.12]{GHN12} essentially verbatim 
with only one notable exception: In \cite[Theorem 3.12]{GHN12}, $\Omega$ is assumed to be a {\it strongly Lipschitz domain} and hence the Stein extension theorem (cf., e.g., \cite[Theorem 5.24]{AF03} or \cite[Theorem 5 in \S VI.3.1]{St70}) is applied to obtain a total extension operator.  In the present case, under the weaker assumptions on $\Omega$, appealing to the Stein extension theorem is {\it not} permitted; so 
instead, we apply \cite[Lemma 3.3]{ABHDR12}, \cite[Lemma 3.4]{ER12} to obtain the extension operator in \eqref{2.19}--\eqref{2.21} (cf. \cite[Remark 3.5]{ER12}).
\end{remark}

Next, we discuss infinitesimal form boundedness for potential coefficients in the critical exponent case in dimensions $n\geq 3$.

\begin{theorem}\lb{t2.10}
Assume Hypotheses \ref{h2.4} and \ref{h2.7}.  Then $V$ is infinitesimally form bounded with respect to $-\Delta_{\Omega,\Gamma}$,
\begin{equation}\lb{2.25}
\big\| |V|^{1/2} f \big\|_{L^2(\Omega)}^2\leq \e \big\|(-\Delta_{\Omega,\Gamma})^{1/2}f \big\|_{L^2(\Omega)}^2 
+ \eta(\e)\|f\|_{L^2(\Omega)}^2,\quad f\in W^{1,2}_{\Gamma}(\Omega),\; \e>0.
\end{equation}
As a result, $V$ is infinitesimally form bounded with respect to $L_{a,\Omega,\Gamma}$,
\begin{equation}\lb{2.26}
\big\| |V|^{1/2} f \big\|_{L^2(\Omega)}^2\leq \e\Re[\gq^{}_{a,\Omega,\Gamma}(f,f)]
+ \widetilde \eta(\e)\|f\|_{L^2(\Omega)}^2,\quad f\in W^{1,2}_{\Gamma}(\Omega),\; \e>0.  
\end{equation}
Here $\eta$ and $\widetilde \eta$ are non-negative functions defined on $(0,\infty)$, generally 
depending on $\Omega$, $n$, and $\Gamma$.
\end{theorem}
\begin{proof}
Again, for simplicity, we put the $L^\infty$-part $V_{\infty}$ of $V$ equal to zero.  The proof is a straightforward modification of the proof of the corresponding result for Lipschitz domains given in \cite[Theorem 3.14\,$(iii)$]{GHN12}, and we present the modified argument here for completeness.  By  Sobolev 
embedding (cf., e.g., \cite[Remark 3.4\,$(ii)$]{ABHDR12}),
\begin{equation}\lb{2.27}
W^{1,2}_{\Gamma}(\Omega) \hookrightarrow L^{2^*}(\Omega), \quad 2^* = 2n/(n-2),
\end{equation}
where ``$\hookrightarrow$'' abbreviates continuous (and dense) embedding, and hence there 
exists a constant $c>0$ such that
\begin{equation}\lb{2.28}
\|f\|_{L^{2^*}(\Omega)}^2\leq c\big(\|\nabla f \|_{L^2(\Omega)^n}^2+\|f\|_{L^2(\Omega)}^2 \big), 
\quad f \in W^{1,2}_{\Gamma}(\Omega).
\end{equation}
Using H\"older's inequality, \eqref{2.28} implies
\begin{equation}\lb{2.29}
\begin{split}
(f,|W|f)_{L^2(\Omega)}\leq c\|W\|_{L^{n/2}(\Omega)}\big(\|\nabla f \|_{L^2(\Omega)^n}^2 
+ \|f\|_{L^2(\Omega)}^2\big),&\\
f \in W^{1,2}_{\Gamma}(\Omega), \; W\in L^{n/2}(\Omega).&
\end{split}
\end{equation}
Next, let $\e>0$ be given.  Since $V\in L^{n/2}(\Omega)$, there exist functions 
$V_{n/2,\e}\in L^{n/2}(\Omega)$ and $V_{\infty,\e}\in L^{\infty}(\Omega)$ with
\begin{equation}\lb{2.30}
\|V_{n/2,\e}\|_{L^{n/2}(\Omega)}\leq\e/c, \quad V(x)=V_{n/2,\e}(x)+V_{\infty,\e}(x) \, 
\text{ for a.e.\ $x\in \Omega$}.
\end{equation}
Applying \eqref{2.29} with $W=V_{n/2,\e}$, one estimates
\begin{align}
\|v f\|_{L^2(\Omega)} &= (f,|V|f)_{L^2(\Omega)} 
\leq (f,[|V_{n/2,\e}|+\|V_{\infty,\e}\|_{L^{\infty}(\Omega)}]f)_{L^2(\Omega)}\no\\
&\leq \e\|\nabla f \|_{L^2(\Omega)^n}^2+\eta(\e)\|f\|_{L^2(\Omega)}^2, 
\quad f \in W^{1,2}_{\Gamma}(\Omega),    \lb{2.31}
\end{align}
with 
\begin{equation}\lb{2.32}
\eta(\e):=\e+\|V_{\infty,\e}\|_{L^{\infty}(\Omega)}.
\end{equation}
Noting the fact that 
\begin{equation}\lb{2.33}
\|\nabla f\|_{L^2(\Omega)^n}^2=\big\|(-\Delta_{\Omega,\Gamma})^{1/2}f \big\|_{L^2(\Omega)}^2, 
\quad f \in W_{\Gamma}^{1,2}(\Omega),
\end{equation}
by the 2nd representation theorem (cf., e.g., \cite[Theorem VI.2.23]{Ka80}), \eqref{2.31} then 
also yields 
\begin{equation}\lb{2.34}
\|vf\|_{L^2(\Omega)}^2\leq \e\|(-\Delta_{\Omega,\Gamma})^{1/2} f\|_{L^2(\Omega)^n}^2 
+ \eta(\e)\|f\|_{L^2(\Omega)}^2,\quad f \in W_{\Gamma}^{1,2}(\Omega).
\end{equation}
Since $\e>0$ was arbitrary and $v=|V|^{1/2}$, \eqref{2.25} follows.

To prove \eqref{2.26}, one notes that the uniform ellipticity condition on $a$ implies
\begin{equation}\lb{2.35}
\|\nabla f\|_{L^2(\Omega)^n}^2 \leq a_1^{-1} \Re[\gq^{}_{a,\Omega,\Gamma}(f,f)], 
\quad f\in W^{1,2}_{\Gamma}(\Omega).
\end{equation}
Taking \eqref{2.35} together with \eqref{2.25} and \eqref{2.33}, one infers that
\begin{equation}\lb{2.36}
\|vf\|_{L^2(\Omega)}^2\leq a_1^{-1}\varepsilon_1 \Re[\gq^{}_{a,\Omega,\Gamma}(f,f)]
+\eta(\varepsilon_1)\|f\|_{L^2(\Omega)}^2, \quad f\in W^{1,2}_{\Gamma}(\Omega),\; \varepsilon_1 > 0.
\end{equation}
The form bound in \eqref{2.26} follows by taking $\varepsilon_1 = a_1 \e$, $\e>0$, in \eqref{2.36}. 
\end{proof}

\section{The Case of Matrix-Valued Divergence Form \\ Elliptic Partial Differential Operators}  \lb{s3}

In this section we consider uniformly elliptic partial differential operators in divergence form 
in the vector-valued context, that is, we will focus on $N\times N$ matrix-valued differential 
expressions $\mathbf{L}$ which act as 
\begin{equation}\lb{3.1}
\mathbf{L} u = - \Bigg(\sum_{j,k=1}^n\partial_j\Bigg(\sum_{\beta = 1}^N 
a^{\alpha,\beta}_{j,k}\partial_k u_\beta\Bigg)\Bigg)_{1\leq\alpha\leq N},\quad u=(u_1,\dots,u_N),  
\end{equation}
and prove our principal result concerning stability of square root domains with respect to additive 
perturbations.

To set the stage, we introduce the following set of hypotheses.

\begin{hypothesis}\lb{h3.1}
Fix $n \in \bbN\backslash\{1\}$, $N\in \bbN$.\\
$(i)$ Assume that $\Omega\subset \bbR^n$ is a non-empty, bounded, open, and connected set 
satisfying the interior corkscrew condition in Definition \ref{d2.1}\,$(i)$.\\
$(ii)$  For each $1\leq \alpha \leq N$, suppose $\Gamma_{\alpha}\subseteq \partial \Omega$ 
is a closed subset of $\partial \Omega$ which is either empty or an $(n-1)$-set, and let 
$\mathbb{G}=(\Gamma_1,\Gamma_2,\ldots,\Gamma_N)$.\\
$(iii)$  Around every point $x\in \ol{\partial \Omega \backslash \cap_{\alpha=1}^N \Gamma_\alpha}$, 
suppose there exists an open neighborhood $U_x \subset \bbR^n$ and a bi-Lipschitz map 
$\Phi_x:U_x\rightarrow (-1,1)^n$ such that 
\begin{align}
\Phi_x(x)&=0,\lb{3.2}\\
\Phi_x(\Omega \cap U_x)&=(-1,1)^{n-1}\times (-1,0),\lb{3.3}\\
\Phi_x(\partial \Omega \cap U_x)&=(-1,1)^{n-1}\times \{0\}.\lb{3.4}
\end{align}
$(iv)$ Define the set
\begin{equation}\lb{3.5}
\cW_{\mathbb{G}}(\Omega) = \prod_{\alpha=1}^N W^{1,2}_{\Gamma_{\alpha}}(\Omega),
\end{equation}
where $W^{1,2}_{\Gamma_{\alpha}}(\Omega)$ is defined as in \eqref{2.8} for each $1\leq \alpha \leq N$,  
suppose that 
\begin{equation} 
a_{j,k}^{\alpha,\beta}\in L^{\infty}(\Omega), \; 1\leq j,k\leq n, \; 1\leq \alpha,\beta\leq N,  
\end{equation} 
and assume that the sesquilinear form in $L^2(\Omega)^N$, 
\begin{equation}\lb{3.6}
\begin{split}
\mathfrak{L}_{a,\Omega,\mathbb{G}}(f,g) = \sum_{j,k=1}^n \sum_{\alpha,\beta=1}^N \int_{\Omega} d^nx \,\ol{(\partial_jf_{\alpha})(x)}a_{j,k}^{\alpha,\beta}(x)(\partial_k g_{\beta})(x),&\\
f,g\in \dom(\mathfrak{L}_{a,\Omega,\mathbb{G}}):=\cW_{\mathbb{G}}(\Omega),&
\end{split}
\end{equation}
satisfies a uniform ellipticity condition of the form, for some $\lambda>0$, 
\begin{equation}\lb{3.7}
\Re[\mathfrak{L}_{a,\Omega,\mathbb{G}}(f,g)]\geq \lambda \sum_{\alpha=1}^N 
\|\nabla f_{\alpha}\|_{L^2(\Omega)^n}^2,\quad f=(f_{\alpha})_{\alpha=1}^N\in \cW_{\mathbb{G}}(\Omega). 
\end{equation}
We denote by $\mathbf{L}_{a,\Omega,\mathbb{G}}$ the m-sectorial operator in $L^2(\Omega)^N$ uniquely associated to the sesquilinear form $\mathfrak{L}_{a,\Omega,\mathbb{G}}$.
\end{hypothesis}

Intuitively, $\mathfrak{L}_{a,\Omega,\mathbb{G}}$ acts on vectors $u = (u_1, u_2,\dots,u_N)$, where each component $u_{\alpha}$ formally satisfies a Dirichlet boundary condition along $\Gamma_{\alpha}$ and a 
Neumann condition along the remainder of the boundary, $\partial \Omega\backslash \Gamma_{\alpha}$, 
$1 \leq \alpha \leq N$ (cf., e.g., \cite[Corollary 4.2]{EHDT13a}). 

\begin{hypothesis}\lb{h3.2}
Let $n \in \bbN\backslash\{1\}$, $N\in \bbN$, and assume that $\Omega\subset \bbR^n$ is 
nonempty and open. Suppose, in addition, that $p>n/2$ and that 
$V_{\alpha,\beta}\in L^{p}(\Omega)+L^{\infty}(\Omega)$ for each $1\leq \alpha,\beta \leq N$. 
\end{hypothesis}

Assuming Hypotheses \ref{h3.1} and \ref{h3.2}, consider the operator of multiplication by the $N\times N$ matrix-valued function $\mathbf{V}=\{V_{\alpha,\beta}\}_{1\leq \alpha,\beta\leq N}$ in $L^2(\Omega)^N$ given by
\begin{equation}\lb{3.8} 
 (\mathbf{V}f)_{\alpha} = \sum_{\beta=1}^N V_{\alpha,\beta} f_{\beta},\quad 1\leq \alpha \leq N, \;  
f\in \dom(\mathbf{V}) = \big\{ f\in L^2(\Omega)^N\, \big|\, \mathbf{V}f \in L^2(\Omega)^N\big\}.
\end{equation}
Next, consider the generalized polar decomposition (cf. \cite{GMMN09}) for $\mathbf{V}$:
\begin{equation}\lb{3.8a}
\mathbf{V} = |\mathbf{V}^*|^{1/2} U |\mathbf{V}|^{1/2},
\end{equation}
where $U$ is an appropriate partial isometry.
The sesquilinear form corresponding to $\mathbf{V}$ is then given by
\begin{equation}\lb{3.9}
\begin{split}
\mathfrak{V}(f,g) = \big(|\mathbf{V}^*|^{1/2} f, U|\mathbf{V}|^{1/2} g\big)_{L^2(\Omega)^N},&\\ 
f,g\in \dom(\mathfrak{V}) = \dom\big(|\mathbf{V}|^{1/2}\big)=\dom\big(|\mathbf{V}^*|^{1/2}\big).&
\end{split}
\end{equation}
With the $L^p(\Omega)$ assumption on each entry $V_{\alpha,\beta}$, $\mathfrak{V}$ is infinitesimally form bounded with respect to $\mathbf{L}_{a,\Omega,\mathbb{G}}$.  In order to prove this, it suffices to consider the case where the $L^{\infty}$-part of each $V_{\alpha,\beta}$ is zero.  In this case, one has the estimate
\begin{align}
|\mathfrak{V}(f,f)|^2 &= \big| \big(|\mathbf{V}^*|^{1/2} f, U|\mathbf{V}|^{1/2} f\big)_{L^2(\Omega)^N}\big|^2\\
&\leq \big|\big(|\mathbf{V}^*|^{1/2} f, |\mathbf{V}|^{1/2} f\big)_{L^2(\Omega)^N}\big|^2\\
&\leq \big\||\mathbf{V}^*|^{1/2} f \big\|_{L^2(\Omega)^N}^2 \big\||\mathbf{V}|^{1/2} f \big\|_{L^2(\Omega)^N}^2\\
&= \int_{\Omega} \big(|\mathbf{V}^*(x)|^{1/2}f(x), |\mathbf{V}^*(x)|^{1/2}f(x)\big)_{\mathbb{C}^N}\, d^nx\no\\
&\quad \times \int_{\Omega} \big(|\mathbf{V}(x)|^{1/2}f(x), |\mathbf{V}(x)|^{1/2}f(x)\big)_{\mathbb{C}^N}\, d^nx\\
&= \int_{\Omega} (f(x),|\mathbf{V}^*(x)|f(x))_{\mathbb{C}^N}\, d^nx \int_{\Omega} (f(x),|\mathbf{V}(x)|f(x))_{\mathbb{C}^N}\, d^nx\\
&\leq \int_{\Omega} \|\mathbf{V}^*(x)\|_{2}\|f(x)\|_{\mathbb{C}^N}^2\, d^nx \int_{\Omega} \|\mathbf{V}(x)\|_{2}\|f(x)\|_{\mathbb{C}^N}^2\, d^nx\\
&=\bigg[\int_{\Omega} \bigg(\sum_{\alpha=1}^N\sum_{\beta=1}^N |V_{\alpha,\beta}(x)|^2 \bigg)^{1/2}\|f(x)\|_{\mathbb{C}^N}^2\, d^nx\bigg]^2\\
&\leq \bigg[\int_{\Omega} W(x) \|f(x)\|_{\mathbb{C}^N}^2\, d^nx\bigg]^2,\quad f\in \dom\big(|\mathbf{V}|^{1/2}\big),\lb{3.17a}
\end{align}
where we have set
\begin{equation}
W(x) = \sum_{\alpha=1}^N\sum_{\beta=1}^N |V_{\alpha,\beta}(x)| \,\text{ for a.e. $x\in \Omega$,}
\end{equation}
and used $\|\, \cdot\,\|_2$ to denote the Hilbert--Schmidt norm of a matrix in $\mathbb{C}^{N\times N}$.  The estimate in \eqref{3.17a} subsequently implies
\begin{align}
|\mathfrak{V}(f,f)|&\leq \int_{\Omega} \big\| W(x)^{1/2}f(x)\|_{\mathbb{C}^N}^2\, d^nx\no\\
&=\sum_{\alpha=1}^N \int_{\Omega}\big|W(x)^{1/2} f_{\alpha}(x)\big|^2\, d^nx\no\\
&=\sum_{\alpha=1}^N \big\|W^{1/2}f_{\alpha} \big\|_{L^2(\Omega)}^2,\quad f\in \dom\big(|\mathbf{V}|^{1/2}\big).\lb{3.20a}
\end{align}
By hypothesis, one infers that $W\in L^p(\Omega)$.  Since $p>n/2$, $W$ is infinitesimally form bounded with respect to $-\Delta_{\Omega,\Gamma_{\alpha}}$ for each $1\leq \alpha \leq N$ (recalling the notational convention $-\Delta_{\Omega,\Gamma}=L_{I_n,\Omega,\Gamma}$ set forth in \eqref{2.12}), with a form bound of the following type:
\begin{equation}\lb{3.21a}
\begin{split}
\big\|W^{1/2}f \big\|_{L^2(\Omega)}^2\leq \varepsilon \big\|(-\Delta_{\Omega,\Gamma_{\alpha}})^{1/2}f \big\|_{L^2(\Omega)} + M_{\alpha}\varepsilon^{-n/(2p-n)}\|f\|_{L^2(\Omega)}^2,&\\
f\in W_{\Gamma_{\alpha}}^{1,2}(\Omega),\, 0<\varepsilon<1,\, 1\leq \alpha\leq N,&
\end{split}
\end{equation}
for appropriate constants $M_{\alpha}>0$, $1\leq \alpha \leq N$.
Setting $M=\max_{1\leq \alpha\leq N}M_{\alpha}$ and applying \eqref{3.21a} to each term of the summation in \eqref{3.20a}, one obtains
\begin{align}
|\mathfrak{V}(f,f)|&\leq \varepsilon \sum_{\alpha=1}^N \big\|(-\Delta_{\Omega,\Gamma_{\alpha}})^{1/2}f_{\alpha} \big\|_{L^2(\Omega)}^2 + M \sum_{\alpha=1}^N\varepsilon^{-n/(2p-n)}\|f_{\alpha}\|_{L^2(\Omega)}^2\no\\
&= \varepsilon \sum_{\alpha=1}^N\|\nabla f_{\alpha}\|_{L^2(\Omega)}^2 + M \varepsilon^{-n/(2p-n)}\|f\|_{L^2(\Omega)^N}^2, \lb{3.22a}\\
& \hspace*{3.68cm} f\in \cW_{\mathbb{G}}(\Omega),\, 0<\varepsilon<1.\no
\end{align}
Finally, applying the uniform ellipticity condition \eqref{3.7} to \eqref{3.22a}, one obtains the form bound,
\begin{align}\lb{3.15}
\begin{split} 
|\mathfrak{V}(f,f)|\leq \lambda^{-1}\varepsilon\, \Re[\mathfrak{L}_{a,\Omega,\mathbb{G}}(f,f)] 
+ M\varepsilon^{-n/(2p-n)}\|f\|_{L^2(\Omega)^N}^2,& \\
f\in \cW_{\mathbb{G}}(\Omega),\; 0<\varepsilon<1.&
\end{split} 
\end{align}
By suitably rescaling $\varepsilon$ throughout \eqref{3.15}, one infers that $\mathbf{V}$ is infinitesimally form bounded with respect to $\mathbf{L}_{a,\Omega,\mathbb{G}}$.  Infinitesimal form boundedness of $\mathbf{V}$ with respect to $\mathbf{L}_{a,\Omega,\mathbb{G}}$ is summarized in the following result.

\begin{theorem}\lb{t3.3}
Assume Hypotheses \ref{h3.1} and \ref{h3.2}.  Then $\mathbf{V}$ is infinitesimally form bounded with respect to $\mathbf{L}_{a,\Omega,\mathbb{G}}$ and there exist constants $M>0$ and $\varepsilon_0>0$ such that
\begin{equation}\lb{3.16}
|\mathfrak{V}(f,f)|\leq \varepsilon\, \Re[\mathfrak{L}_{a,\Omega,\mathbb{G}}(f,f)] 
+ M\varepsilon^{-n/(2p-n)}\|f\|_{L^2(\Omega)^N}^2, \quad 
f\in \cW_{\mathbb{G}}(\Omega),\; 0<\varepsilon<\varepsilon_0.
\end{equation}
\end{theorem}

In view of Theorem \ref{t3.3}, the form sum $\mathbf{L}_{a,\Omega,\mathbb{G}}+_{\mathfrak{q}}\mathbf{V}$ is well-defined and represents an m-sectorial operator in $L^2(\Omega)^N$.  Our next result extends stability of square root domains to $\mathbf{L}_{a,\Omega,\mathbb{G}}$ and $\mathbf{L}_{a,\Omega,\mathbb{G}}+_{\mathfrak{q}}\mathbf{V}$.

\begin{theorem}\lb{t3.4}
Assume Hypotheses \ref{h3.1} and \ref{h3.2} and let $\mathbf{V}$ denote the operator of component-wise multiplication in $L^2(\Omega)^N$ defined in \eqref{3.8}.  Then
\begin{equation}\lb{3.17}
\begin{split}
&\dom\big((\mathbf{L}_{a,\Omega,\mathbb{G}}+_{\mathfrak{q}}\mathbf{V})^{1/2}\big) =  \dom\big(((\mathbf{L}_{a,\Omega,\mathbb{G}}+_{\mathfrak{q}}\mathbf{V})^*)^{1/2}\big)\\
&\quad = \dom\big(\mathbf{L}_{a,\Omega,\mathbb{G}}^{1/2}\big) = \dom\big((\mathbf{L}_{a,\Omega,\mathbb{G}}^*)^{1/2}\big) = \cW_{\mathbb{G}}(\Omega).
\end{split}
\end{equation}
\end{theorem}
\begin{proof}
Let $\mathbf{V}$ be factored into the form $\mathbf{V} = \mathbf{B}^* \mathbf{A}$
\begin{align}
\mathbf{A} = U|\mathbf{V}|^{1/2}, \quad \mathbf{B} = |\mathbf{V}^*|^{1/2}, \quad \dom(\mathbf{A}) = \dom(\mathbf{B}) = \dom(|\mathbf{V}|^{1/2}),
\end{align}
according to the generalized polar decomposition in \eqref{3.8a}.
One observes that $\cW_{\mathbb{G}}(\Omega)\subset \dom(\mathbf{A})=\dom(\mathbf{B})$.  Therefore, \cite[Theorem 9.2]{EHDT13a} implies
\begin{align}
\dom(\mathbf{A})\supseteq \dom\big(\mathbf{L}_{a,\Omega,\mathbb{G}}^{1/2}\big),\quad 
\dom(\mathbf{B})\supseteq \dom\big((\mathbf{L}_{a,\Omega,\mathbb{G}}^*)^{1/2}\big).\lb{3.19}
\end{align}

Next, let $\mathbf{D}_{\Omega,\mathbb{G}}$ denote the non-negative self-adjoint operator uniquely associated to the sesquilinear form
\begin{equation}\lb{3.20}
\begin{split}
\mathfrak{D}_{\Omega,\mathbb{G}}(f,g)=\sum_{j,k=1}^n \sum_{\alpha,\beta=1}^N \int_{\Omega} \ol{(\partial_jf_{\alpha})(x)}\delta_{j,k}\delta_{\alpha,\beta}(\partial_k g_{\beta})(x)\, d^nx,&\\
f,g\in \dom(\mathfrak{D}_{\Omega,\mathbb{G}}):=\cW_{\mathbb{G}}(\Omega),&
\end{split}
\end{equation}
where $\delta_{j,k}$ denotes the Kronecker delta symbol.  (One notes that \eqref{3.20} is 
simply \eqref{3.6} with tensor coefficients $a_{j,k}^{\alpha,\beta}=\delta_{j,k}\delta_{\alpha,\beta}$, 
$1\leq j,k \leq n$, $1\leq \alpha, \beta\leq N$.)  Then (cf., e.g., the discussion preceding 
\cite[Theorem 9.2]{EHDT13a})
\begin{equation}\lb{3.21}
(\mathbf{D}_{\Omega,\mathbb{G}}f)_{\alpha} = -\Delta_{\Omega,\Gamma_{\alpha}}f_{\alpha}, 
\quad 1\leq \alpha \leq N,\; f\in \dom(\mathbf{D}_{\Omega,\mathbb{G}})
=\prod_{\beta=1}^N \dom(-\Delta_{\Omega,\Gamma_{\beta}}),
\end{equation}
where we have used the notational convention 
$-\Delta_{\Omega,\Gamma_{\alpha}}=L_{I_n,\Omega,\Gamma_{\alpha}}$ set forth in \eqref{2.12}. 
In addition (cf. the discussion in the proof to \cite[Theorem 9.2]{EHDT13a}),
\begin{equation}\lb{3.22}
\big(\mathbf{D}_{\Omega,\mathbb{G}}^{1/2}f\big)_{\alpha} 
= (-\Delta_{\Omega,\Gamma_{\alpha}})^{1/2}f_{\alpha},\quad 1\leq \alpha \leq N,\; 
f\in \dom\big(\mathbf{D}_{\Omega,\mathbb{G}}^{1/2}\big)=\cW_{\mathbb{G}}(\Omega).
\end{equation}
As a result of \eqref{3.22}, the bound in \eqref{3.22a} actually implies
\begin{align}\lb{3.23}
\begin{split}
\|\mathbf{A}f\|_{L^2(\Omega)^N}^2\leq  \varepsilon \big\|\mathbf{D}_{\Omega,\mathbb{G}}^{1/2} f\big\|_{L^2(\Omega)^N}^2+M\varepsilon^{-n/(2p-n)}\|f\|_{L^2(\Omega)^N}^2,& \\
f\in \cW_{\mathbb{G}}(\Omega),\; 0<\varepsilon<1.&
\end{split}
\end{align}
Hence, \cite[Lemma 2.12]{GHN12} guarantees the existence of constants $M_1>0$, $q>0$, and 
$E_0\geq 1$ such that  
\begin{equation}\lb{3.24}
\big\|\mathbf{A}(\mathbf{D}_{\Omega,\mathbb{G}}+EI_{L^2(\Omega)^N})^{-1/2}\big\|_{\cB(L^2(\Omega)^N)}\leq M_1E^{-q},\quad E>E_0.
\end{equation}
Since $\dom\big(\mathbf{L}_{a,\Omega,\mathbb{G}}^{1/2}\big)=\dom\big(\mathbf{D}_{\Omega,\mathbb{G}}^{1/2}\big)$, \cite[Lemma 2.11]{GHN12} yields the existence of a constant $C>0$ such that 
\begin{equation}\lb{3.25}
\sup_{E\geq 1}\big\|(\mathbf{L}_{a,\Omega,\mathbb{G}}+EI_{L^2(\Omega)^N})^{1/2} (\mathbf{D}_{\Omega,\mathbb{G}}+EI_{L^2(\Omega)^N})^{-1/2} \big\|_{\cB(L^2(\Omega)^N)}\leq C.
\end{equation}
The estimates in \eqref{3.24} and \eqref{3.25} imply
\begin{equation}\lb{3.26}
\big\|\mathbf{A}(\mathbf{L}_{a,\Omega,\mathbb{G}}+EI_{L^2(\Omega)^N})^{-1/2}\big\|_{\cB(L^2(\Omega)^N)}\leq {\widehat M}_1 E^{-q},\quad E>E_0,
\end{equation}
for an appropriate constant ${\widehat M}_1>0$.  A similar argument involving adjoints can be used to show
\begin{equation}\lb{3.27}
\big\|\ol{(\mathbf{L}_{a,\Omega,\mathbb{G}}+EI_{L^2(\Omega)^N})^{-1/2}\mathbf{B}^*}\big\|_{\cB(L^2(\Omega)^N)}\leq {\widehat M}_2E^{-q},\quad E>E_0,
\end{equation}
for an appropriate constant ${\widehat M}_2 > 0$.

Finally, in light of \eqref{3.19}, \eqref{3.26}, \eqref{3.27}, and the fact that (cf. \cite[Theorem 9.2]{EHDT13a})
\begin{equation}\lb{3.28}
\dom\big(\mathbf{L}_{a,\Omega,\mathbb{G}}^{1/2}\big)=\dom\big((\mathbf{L}_{a,\Omega,\mathbb{G}}^*)^{1/2}\big),
\end{equation}
the string of equalities in \eqref{3.17} follows from an application of \cite[Corollary 2.7]{GHN12}.  
We note that \cite[Hypothesis 2.1\,$(iii)$]{GHN12} holds in the present setting since
\begin{align}
&\Big\|\ol{\mathbf{A}(\mathbf{L}_{a,\Omega,\mathbb{G}}+EI_{L^2(\Omega)^N})^{-1}\mathbf{B}^*}\Big\|_{\cB(L^2(\Omega)^N)}\no\\
&\quad \leq \big\|\mathbf{A}(\mathbf{L}_{a,\Omega,\mathbb{G}}+EI_{L^2(\Omega)^N})^{-1/2}\big\|_{\cB(L^2(\Omega)^N)}\lb{3.29}\\
&\qquad \times \big\|\ol{(\mathbf{L}_{a,\Omega,\mathbb{G}}+EI_{L^2(\Omega)^N})^{-1/2}\mathbf{B}^*}\big\|_{\cB(L^2(\Omega)^N)},\quad E>0,\no
\end{align}
and the estimates in \eqref{3.26}, \eqref{3.27} yield decay to zero as $E\uparrow \infty$ of the factors on the right-hand side of \eqref{3.29}.
\end{proof}

\noindent


\end{document}